\documentclass[12pt,reqno]{rmmcart}
\usepackage{amsmath,amssymb,amscd,graphicx}
\usepackage[all]{xy}
\usepackage[latin1]{inputenc}
\usepackage[T1]{fontenc}
\usepackage{hyperref}
\usepackage{ifthen}
\usepackage{subsupscripts}

\newboolean{workmode}
\setboolean{workmode}{false}

\newboolean{sections}
\setboolean{sections}{true}

\input xy
\xyoption{all}

\def\eoe{\unskip\ \hglue0mm\hfill$\lozenge$\smallskip\goodbreak}

\makeatletter
\def\th@plain{%
  \thm@notefont{}
  \itshape 
}
\def\th@definition{%
  \thm@notefont{}
  \normalfont 
}
\makeatother

\newcommand{\CC}{{\mathbb{C}}}  
\newcommand{\NN}{{\mathbb{N}}}  
\newcommand{\RR}{{\mathbb{R}}}  
\renewcommand{\SS}{{\mathbb{S}}} 
\newcommand{\ZZ}{{\mathbb{Z}}}  

\newcommand{\Der}{{\operatorname{Der}}}  
\newcommand{\pr}{{\operatorname{pr}}} 
\newcommand{\Stab}{{\operatorname{Stab}}} 
\newcommand{\Vect}{{\operatorname{vect}}}  

\newcommand{\GL}{{\operatorname{GL}}}  

\newcommand{\CIN}{{C^\infty}}   
\newcommand{\toto}{{~\rightrightarrows~}} 

\newcommand{\hypref}[2]{{#2~\ref{#1}}}

\newcommand{\ifwork}[1]{\ifthenelse{\boolean{workmode}}{#1}{}}
\newcommand{\comment}[1]{}
\newcommand{\mute}[1]{}
\newcommand{\printname}[1]{}

\ifwork{
\renewcommand{\comment}[1]{{\marginpar{*}\ \scriptsize{#1}\ }}
\renewcommand{\printname}[1]
    {\smash{\makebox[0pt]{\hspace{-1.0in}\raisebox{8pt}{\tiny #1}}}}
}

\newcommand{\labell}[1] {\label{#1} \printname{#1}}

\newcommand{\ifsection}[2]{\ifthenelse{\boolean{sections}}{#1}{#2}}

\theoremstyle{plain}
\ifsection{
    \newtheorem{theorem}{Theorem}[section]
}
{
    \newtheorem{theorem}{Theorem}
}

\newtheorem{proposition}[theorem]{Proposition}

\newtheorem{corollary}[theorem]{Corollary}
\newtheorem{lemma}[theorem]{Lemma}

\theoremstyle{definition}
\newtheorem{definition}[theorem]{Definition}
\newtheorem{example}[theorem]{Example}
\newtheorem{remark}[theorem]{Remark}

\def\eoe{\unskip\ \hglue0mm\hfill$\diamond$\smallskip\goodbreak} 

\newtheorem*{main}{Main Theorem}
\newtheorem*{theoremA}{Theorem A}
\newtheorem*{theoremB}{Theorem B}

\newcommand{\diffeol}{\mathbf{Diffeol}} 
\newcommand{\diffsp}{\mathbf{DiffSp}} 
\newcommand{\orb}{\mathbf{Orb}} 
\newcommand{\cod}{\operatorname{cod}} 

\title{The Differential Structure of an Orbifold}

\author{Jordan Watts}
\address{Department of Mathematics, University of Colorado at Boulder, CO, 80309, USA}
\email{jordan.watts@colorado.edu}

\date{September 28, 2015}

\thanks{2010 AMS \emph{Mathematics subject classification}. Primary: 57R18, Secondary: 58A40.}

\begin{document}

\begin{abstract}
We prove that the underlying set of an orbifold equipped with the ring of smooth real-valued functions completely determines the orbifold atlas.  Consequently, we obtain an essentially injective functor from orbifolds to differential spaces.
\end{abstract}

\maketitle

\section{Introduction}\labell{s:intro}

Consider an (effective) orbifold $X$; that is, in particular, a space that locally is the quotient of a smooth manifold by an effective finite Lie group action.  The family of all ``smooth'' functions consists of real-valued functions on $X$ that locally lift to these manifolds as smooth functions invariant under the finite group actions.  This family is an example of a (Sikorski) differential structure (see \hypref{d:diff space}{Definition}).  The purpose of this paper is to prove the following theorem.

\begin{main}\labell{t:main}
Given an orbifold, its orbifold atlas can be constructed out of invariants of the differential structure.
\end{main}

This result can be tailored to be in the form of a functor from the ``category of orbifolds'' to differential spaces which is essentially injective on objects.  Of course, the ``category of orbifolds'' has a number of different definitions, depending on one's perspective.  There is the classical ``category'' defined by Satake \cite{satake} and further developed by Thurston \cite{thurston} and Haefliger \cite{haefliger}.  There are subtle differences between the definitions given by Satake and Haefliger, but we choose not to expand upon these here. (\cite{IZKZ} does deal with this subtlety, however).   There is also the category of effective proper \'etale Lie groupoids (with various choices for the arrows), or the corresponding 2-subcategory of geometric stacks.  See, for example, \cite{HS},\cite{lerman},\cite{moerdijk},\cite{MM},\cite{MP},\cite{pronk}. Choosing to use the weak 2-category of Lie groupoids with bibundles as arrows, we have:

\begin{theoremA}\labell{t:gpds}
There is a functor $F$ from the weak 2-category of effective proper \'etale Lie groupoids with bibundles as arrows to differential spaces that is essentially injective on objects.
\end{theoremA}

Here, ``essentially injective'' means that given two objects $G$ and $H$ such that $F(G)\cong F(H)$, we have $G\simeq H$, where in this case $\simeq$ means Morita equivalent.  It should be noted that this functor is neither full nor faithful (see \hypref{x:not full nor faithful}{Example}, which consists of Examples 24 and 25 of \cite{IZKZ}).  The other modifications of the category of Lie groupoids (including stacks) mentioned in the references listed above will yield a similar theorem.

In \cite{IZKZ}, Iglesias-Zemmour, Karshon, and Zadka define the notion of a ``diffeological orbifold'', and show that this agrees with the classical definitions as found in \cite{satake} and \cite{haefliger}.  Using this, we show:

\begin{theoremB}\labell{t:IZKZ}
There is a functor $G$ from the weak 2-category of effective proper \'etale Lie groupoids with bibundles as arrows to diffeological spaces that is essentially injective on objects.
\end{theoremB}

We give two proofs of this.  The first uses the fact that $G$ is the restriction of a more general functor from the weak 2-category of Lie groupoids to diffeological spaces introduced in \cite[Section 4]{watts-gpds}.  The essential injectivity follows immediately from the work of Iglesias-Zemmour, Karshon, and Zadka.  The second proof of the essential injectivity of \hyperref[t:IZKZ]{Theorem B} uses the fact that the functor $F$ in \hyperref[t:gpds]{Theorem A} factors as $\mathbf{\Phi}\circ G$, where $\mathbf{\Phi}$ is a faithful functor from diffeological spaces to differential spaces sending a diffeological space to its underlying set equipped with the ring of diffeologically smooth functions (see \cite[Chapter 2]{watts} and \cite{BIZKW}).  Both \hyperref[t:gpds]{Theorem A} and \hyperref[t:IZKZ]{Theorem B} rely on a known correspondence between effective proper \'etale Lie groupoids and orbifolds in the classical sense (see \hypref{r:groupoid}{Remark}).  For more on the relationship between Lie groupoids and diffeological spaces, see \cite{KZ} and \cite{watts-gpds}.

The central idea behind the proof of the Main Theorem is as follows.  To reconstruct the orbifold, one needs three ingredients: the topology, the orbifold stratification, and the order of points at codimension-2 strata (see \hypref{d:point order}{Definition}).  We shows that all three of these are invariants of the differential structure of an orbifold.  This fact for the topology and, locally, the stratification are more-or-less already known.  We give a local-to-global argument for the stratification in \hypref{t:orbits}{Theorem}, and use codimensions of germs of functions (similar to Milnor numbers) to obtain that the order of a point is an invariant of the differential structure in \hypref{p:milnor}{Proposition}.  From here, a method proved by Haefliger and Ngoc Du \cite{HND} is used to reconstruct the local isotropy groups, and an argument by induction on the dimension of the orbifold is used to reconstruct the charts.

Differential spaces were introduced by Sikorski in 1967 (\cite{sikorski1}, \cite{sikorski2}), and the theory was further developed by many since then, often times under different names (see, for example, Schwarz \cite{schwarz}, \'Sniatycki \cite{sniatycki}, and Aronszajn \cite{aronszajn}).  When dealing with quotient spaces such as orbifolds, the differential structure is induced by the diffeology, which in turn is induced by the corresponding Lie groupoid/stack.  Thus the differential structure is a fairly weak structure in this setting.  It is equivalent to the corresponding Fr\"olicher space structure (see \cite{cherenack}).

The fact that \hyperref[t:gpds]{Theorem A} is true given \hyperref[t:IZKZ]{Theorem B} is \emph{a priori} unexpected.  Indeed, consider orbifolds of the form $X=\RR^n/\Gamma$.  As mentioned above, the differential structure on $X$ is induced by the diffeological structure, but this relationship is definitely not one-to-one when looking at general group actions.  In fact, the differential structure on $\RR^n/O(n)$ is independent of $n$, while the diffeology is dependent on $n$ (see \hypref{x:rotations}{Example}).  What we can conclude from this is that there is something special about the underlying (local) semi-algebraic structure of an orbifold (equipped with its natural differential structure) that allows us to reconstruct the original orbifold atlas.

This paper is broken down as follows.  \hypref{s:differential spaces}{Section} reviews the relevant theory of differential spaces.  \hypref{s:orb diff str}{Section} reviews the definition of an orbifold, defines its differential structure, and develops properties of it.  \hypref{s:strat}{Section} discusses the natural stratification of an orbifold, and here we prove that this stratification is an invariant of the differential structure (\hypref{c:strat}{Corollary}).  In \hypref{s:recover charts}{Section} we prove that the order of a point is an invariant of the differential structure (\hypref{t:point order}{Theorem}), reconstruct the isotropy groups (\hypref{t:HND}{Theorem}), and reconstruct the charts (the proof of the Main Theorem).  \hypref{s:functor}{Section} contains the proof of \hyperref[t:gpds]{Theorem A}.  \hypref{s:diffeology}{Section} contains both proofs of \hyperref[t:IZKZ]{Theorem B}.

Similar unpublished work for orbifolds whose isotropy groups are reflection-free or completely generated by reflections has been done by Moshe Zadka (see the introduction of \cite{IZKZ}), although this is not available as a preprint, and the author has not seen it.

The author wishes to thank Brent Pym, who made the author aware of Milnor numbers, which saved him from ``reinventing the wheel'' (or perhaps something less round).  The author also wishes to thank the University of Illinois at Urbana-Champaign for their support during the development and writing of this paper.

\section{Review of Differential Spaces}\labell{s:differential spaces}

In this section we review the basics of differential spaces, and give relevant examples.  For a more detailed presentation of differential spaces, see \cite{sniatycki} or Section 2.2 of \cite{watts}.

\begin{definition}[Differential Space]\labell{d:diff space}
Let $X$ be a set.  A \emph{(Sikorski) differential structure}
on $X$ is a family
$\mathcal{F}$ of real-valued functions on $X$ satisfying the following two conditions:
\begin{enumerate}
\item\labell{i:smooth compatibility} \textbf{(Smooth Compatibility)} For any positive integer $k$, functions
$f_1,...,f_k\in\mathcal{F}$, and $g\in\CIN(\RR^k)$, the composition
$g(f_1,...,f_k)$ is contained in $\mathcal{F}$.

\item\labell{i:locality} \textbf{(Locality)} Equip $X$ with the weakest topology for which each $f\in\mathcal{F}$ is continuous.  Let $f \colon X\to\RR$ be a function such that there exist an open cover $\{U_\alpha\}$ of $X$ and for each $\alpha$, a function $g_\alpha\in\mathcal{F}$ satisfying $$f|_{U_\alpha}=g_\alpha|_{U_\alpha}.$$  Then $f\in\mathcal{F}$.
\end{enumerate}
The topology in the Locality Condition is called the \emph{functional topology} (or \emph{initial topology}) induced by $\mathcal{F}$.  A set $X$ equipped with a differential structure $\mathcal{F}$ is called
a \emph{(Sikorski) differential space} and is denoted by $(X,\mathcal{F})$.
\end{definition}

\begin{definition}[Functionally Smooth Map]\labell{d:smooth map}
Let $(X,\mathcal{F}_X)$ and $(Y,\mathcal{F}_Y)$ be two differential
spaces.  A map $F \colon X\to Y$ is \emph{functionally smooth}
if $F^*\mathcal{F}_Y\subseteq\mathcal{F}_X$.
The map $F$ is called a \emph{functional diffeomorphism}
if it is a bijection and both it and its inverse are smooth.
\end{definition}

\begin{remark}
Differential spaces along with functionally smooth maps form a category, which we denote by $\diffsp$.  Except for where it would be ambiguous, ``functional'' and ``functionally'' will be dropped henceforth.
\end{remark}

\begin{definition}[Differential Subspace]\labell{d:diff subspace}
Let $(X,\mathcal{F})$ be a differential space, and let $Y\subseteq X$ be
any subset.  Then $Y$ comes equipped with a differential structure $\mathcal{F}_Y$ induced by $\mathcal{F}$ as follows.
A function $f\in\mathcal{F}_Y$ if and only if there is a covering $\{U_\alpha\}$ of $Y$ by open sets of $X$ such that for each $\alpha$, there exists $g_\alpha\in\mathcal{F}$ satisfying $$f|_{U_\alpha\cap Y}=g_\alpha|_{U_\alpha\cap Y}.$$  We call $(Y,\mathcal{F}_Y)$ a \emph{differential subspace} of $X$.  The functional topology on $Y$ induced by $\mathcal{F}_Y$ coincides with the subspace topology on $Y$ (see \cite[Lemma 2.28]{watts}).  If $Y$ is a closed differential subspace of $\RR^n$, then $\mathcal{F}_Y$ is the set of restrictions of smooth functions on $\RR^n$ to $Y$ (see \cite[Proposition 2.36]{watts}).
\end{definition}

\begin{definition}[Subcartesian Space]\labell{d:subcart}
A \emph{subcartesian space} is a paracompact, second-countable, Hausdorff differential space $(S,\CIN(S))$ with an open cover $\{U_\alpha\}$ such that for each $\alpha$, there exist $n_\alpha\in\NN$ and a diffeomorphism $\varphi_\alpha:U_\alpha\to\tilde{U}_\alpha\subseteq\RR^{n_\alpha}$ onto a differential subspace $\tilde{U}_\alpha$ of $\RR^{n_\alpha}$.
\end{definition}

\begin{example}[Some Semi-Algebraic Varieties]\labell{x:varieties1}
Let $k$ be a positive integer.  Define $$S_k:=\{(x,y)\in\RR^2\mid y^2-x^k=0,~x\geq 0\}.$$  Then $S_k$ is a closed differential subspace of $\RR^2$, with a differential structure given by all real-valued functions that extend to smooth functions on $\RR^2$.

Similarly, define $$C_k:=\{(x,y,z)\mid x^2+y^2=z^k,~z\geq 0\}.$$  Then $C_k$ is a closed differential subspace of $\RR^3$, and hence its differential structure is given by restrictions of smooth functions on $\RR^3$. We will encounter these spaces again in later examples.  \eoe
\end{example}

\begin{definition}[Quotient Differential Structure]\labell{d:diff quotient}
Let $(X,\mathcal{F})$ be a differential space, let $\sim$ be an
equivalence relation on $X$, and let $\pi:X \to X/\!\!\sim$
be the quotient map.
Then $X/\!\!\sim$ obtains a differential
structure $\mathcal{F}_\sim$, called the \emph{quotient differential structure},
comprising all functions $f:X/\!\sim\,\to\RR$ each of whose pullback by $\pi$ is in $\mathcal{F}$.  In general, the functional topology generated by $\mathcal{F}_\sim$ is coarser than the quotient topology.
\end{definition}

\begin{example}[Orbit Space]\labell{x:orbitspace}
Let $K$ be a Lie group acting on a manifold $M$.  Then the quotient differential structure on the orbit space $M/K$ consists of all functions each of which pulls back to a $K$-invariant smooth function on $M$.

Continuing this example, if $K$ is a compact group (or if $K$ acts on $M$ properly), then $M/K$ is in fact a subcartesian space.  Indeed, by the local nature of a subcartesian space and the Slice Theorem (\cite{koszul}, \cite{palais}), it is enough to consider $K$ as a subgroup of $O(n)$ acting on $\RR^n$.  By a theorem of Schwarz \cite{schwarz}, the Hilbert map $\sigma=(\sigma_1,...,\sigma_k):\RR^n\to\RR^k$, where $\sigma_1,...,\sigma_k$ is a minimal generating set of the ring of $K$-invariant polynomials, descends to a proper topological embedding of $\RR^n/K$ as a closed subset of $\RR^k$.  Moreover, $\sigma^*(\CIN(\RR^k))=\CIN(\RR^n)^K$, which implies that the quotient differential structure on $\RR^n/K$ is equal to the subcartesian structure induced by $\RR^k$.
\eoe
\end{example}

\section{Orbifolds and their Differential Structures}\labell{s:orb diff str}

We begin this section with the classical definition of an orbifold, based on the presentation in Section 1 of Moerdijk-Pronk \cite{MP}.  We then discuss its natural differential structure.

\begin{definition}[(Effective) Orbifold]\labell{d:orbifold}
Let $X$ be a Hausdorff, paracompact, second-countable topological space.  Fix a non-negative integer $n$.
\begin{enumerate}
\item An $n$-dimension \emph{orbifold chart} on $X$ is a triple $(U,\Gamma,\phi)$ where $U\subseteq\RR^n$ is an open subset, $\Gamma$ is a finite group of diffeomorphisms of $U$, and $\phi$ is a $\Gamma$-invariant map $\phi:U\to X$ that induces a homeomorphism $U/\Gamma\to\varphi(U)$.

\item An \emph{embedding} $\lambda:(U,\Gamma,\phi)\to(V,\Delta,\psi)$ between two charts is a smooth embedding $\lambda:U\to V$ such that $\psi\circ\lambda=\phi$.

\item An $n$-dimensional \emph{orbifold atlas} on $X$ is a family $\mathcal{U}$ of $n$-dimensional orbifold charts that cover $X$ and are locally compatible.  This last condition means that for any two charts $(U,\Gamma,\phi)$ and $( V,\Delta,\psi)$ in $\mathcal{U}$ there is a family of charts $\{( W_\alpha,\Gamma_\alpha,\chi_\alpha)\}$ with embeddings $( W_\alpha,\Gamma_\alpha,\chi_\alpha)\to(U,\Gamma,\phi)$ and $( W_\alpha,\Gamma_\alpha,\chi_\alpha)\to( V,\Delta,\psi)$ for each $\alpha$, and the collection $\{\chi_\alpha( W_\alpha)\}$ forms an open cover of $\phi(U)\cap\psi( V)$.

\item An orbifold atlas $\mathcal{U}$ \emph{refines} another orbifold atlas $\mathcal{V}$ if for any chart in $\mathcal{U}$, there is an embedding of the chart into a chart of $\mathcal{V}$.  If there exists a common refinement of $\mathcal{U}$ and $\mathcal{V}$, then we say that the two atlases are \emph{equivalent}. This forms an equivalence relation on all atlases of $X$.  Each such equivalence class is represented by a maximal atlas.

\item An \emph{(effective) orbifold} $(X,\mathcal{U})$ of dimension $n$ is a Hausdorff, paracompact, second-countable space $X$ equipped with a maximal $n$-dimensional atlas $\mathcal{U}$.

\item Let $(X,\mathcal{U})$ and $(Y,\mathcal{V})$ be orbifolds.  Then a map $F:X\to Y$ is \emph{orbifold smooth} if for any $x\in X$, there exist charts $(U,\Gamma,\phi)$ about $x$ and $( V,\Delta,\psi)$ about $F(x)$ such that $F(\phi(U))\subseteq\psi( V)$ and there exists a smooth map $\tilde{F}:U\to V$ such that $\psi\circ\tilde{F}=F\circ\phi$.  If $F$ is orbifold smooth and invertible with orbifold smooth inverse, then $F$ is an \emph{orbifold diffeomorphism}.
\end{enumerate}
\end{definition}

\begin{remark}\labell{r:orbifold}
\noindent
\begin{enumerate}
\item \labell{i:locally compact} The topology on an orbifold is locally compact, and since it is Hausdorff and second-countable it follows that the topology is also normal.
\item \labell{i:group monomorphism} Let $X$ be an orbifold, and $\lambda:(W,\Delta,\chi)\to(U,\Gamma,\phi)$ an embedding of charts.  Then $\lambda$ induces a group monomorphism $\bar{\lambda}:\Delta\to\Gamma$ such that for any $w\in W$ and $\delta\in \Delta$, $$\lambda(\delta\cdot w)=\bar{\lambda}(\delta)\cdot\lambda(w).$$  Moreover, if $\gamma\in\Gamma$ such that $\lambda(U)\cap \gamma\cdot\lambda(U)\neq\emptyset$, then $\gamma$ is in the image of $\bar{\lambda}$.  In particular, for any $w\in W$, $\lambda$ induces a group isomorphism between the stabiliser of $w$ and that of $\lambda(w)$ (see \cite[Appendix]{MP}).  A similar statement appears as Lemma 17 of \cite{IZKZ}.
\end{enumerate}
\end{remark}

\begin{example}[Reflections and Rotations in the Plane - Part I]\labell{x:varieties2}
Let $D_k$ be the dihedral group of order $2k$. It is generated by $\beta_1$ and $\beta_2$, both of which have order 2, and such that $(\beta_2\beta_1)^k$ is the identity.  $\beta_1$ acts on $\CC\cong\RR^2$ by conjugation $z\mapsto\bar{z}$, and $\beta_2$ by $z\mapsto e^{2\pi i/k}\bar{z}$.  The resulting orbit space $\RR^2/D_k$ is an example of an orbifold.

Similarly, let $\ZZ_k$ be the cyclic group of order $k$.  It is generated by $\alpha$, which has order $k$. It acts on $\CC\cong\RR^2$ by $z\mapsto e^{2\pi i/k}z$.  We obtain the orbifold $\RR^2/\ZZ_k$.
\eoe
\end{example}

\begin{theorem}[A Theorem of Leonardo di Vinci - Finite Group Actions on the Plane]\labell{t:LDV}
Let $\Gamma\subset O(2)$ be a finite group acting orthogonally on the plane.  Then $\Gamma$ is isomorphic as a group to a dihedral group $D_k$ or to a cyclic group $\ZZ_k$.
\end{theorem}

\begin{proof}
The cyclic and dihedral groups are the only finite Lie subgroups of $O(2)$.  See for example pages 66 or 99 of \cite{weyl} for a reference attributing this discovery to di Vinci.
\end{proof}

\begin{remark}\labell{r:LDV}
Due to \hypref{t:LDV}{Theorem} and the fact that any finite linear group action on the plane can be transformed equivariant-diffeomorphically into an orthogonal group action (one can always construct an invariant metric) we conclude that any 2-dimensional orbifold locally looks like $\RR^2/D_k$ or $\RR^2/\ZZ_k$ for some $k$.
\end{remark}

\begin{definition}[Isotropy Group]\labell{d:isotropy group}
Let $X$ be an orbifold of dimension $n$ and let $x\in X$.  Then an \emph{isotropy group} of $X$ at $x$ is a finite subgroup $\Gamma_x$ of $\GL(\RR^n)$ such that there exists a chart $(\RR^n,\Gamma_x,\phi)$ satisfying $\phi(0)=x$.
\end{definition}

\begin{remark}\labell{r:isotropy group}
An isotropy group exists at every point $x\in X$, and can be obtained using the Slice Theorem.  It is unique up to conjugation in $\GL(\RR^n)$ (see \cite[pages 39-40]{MM}).  Moreover, we may assume that $\Gamma_x\in O(n)$ if needed.
\end{remark}

\begin{definition}[Differential Structure on an Orbifold]\labell{d:diff orbifold}
Let $X$ be an $n$-dimensional orbifold.  Then the \emph{orbifold differential structure} $\CIN(X)$ on $X$ is given by real-valued functions $f:X\to\RR$ satisfying the following: given a chart $(U,\Gamma,\phi)$ of $X$, there exists a smooth $\Gamma$-invariant function $g_U:U\to\RR$ such that $g_U=\phi^*f$.
\end{definition}

\begin{proposition}[Properties of the Orbifold Differential Structure]\labell{p:diff orb}
Let $X$ be an orbifold.
\begin{enumerate}
\item \labell{i:topology} The corresponding functional topology on $X$ equals the orbifold topology.
\item \labell{i:rings} $\CIN(X)$ equals the ring of orbifold smooth functions.
\item \labell{i:orb is subcart} $(X,\CIN(X))$ is subcartesian.
\end{enumerate}
\end{proposition}

\begin{proof}
\noindent
\begin{enumerate}
\item A basis for the topology on $X$ induced by its orbifold structure is given by the union over all charts $(U,\Gamma,\phi)$ of each quotient topology on $\phi(U)$.  But by \hypref{x:orbitspace}{Example} and \hypref{d:diff subspace}{Definition}, this is also a basis for the topology induced by $\CIN(X)$.
\item This is immediate from the definitions.
\item This is a direct consequence of \hypref{x:orbitspace}{Example} and \hypref{d:diff orbifold}{Definition}.
\end{enumerate}
\end{proof}

\begin{remark}\labell{r:diff orb}
It follows from \hypref{p:diff orb}{Proposition} that the orbifold differential structure only depends on the natural differential structure of the (local) semi-algebraic variety underlying the orbifold.
\end{remark}

\begin{example}[Reflections and Rotations in the Plane - Part II]\labell{x:varieties3}
Continuing example \hypref{x:varieties2}{Example}, a minimal generating set for the ring of $D_k$-invariant real polynomials on $\CC\cong\RR^2$ is given by $\{\delta_1,\delta_2\}$ where $\delta_1(z)=|z|^2$ and $\delta_2(z)=\Re(z^k)$.  The resulting orbifold embeds into $\RR^2$ as the semi-algebraic variety $$R_k:=\{(s,t)\mid t^2\leq s^{k},~s\geq 0\}.$$

Similarly, a minimal generating set for the ring of $\ZZ_k$-invariant real polynomials on $\CC\cong\RR^2$ is given by $\{\sigma_1,\sigma_2,\sigma_3\}$ where $\sigma_1=\Re(z^k)$, $\sigma_2=\Im(z^k)$, and $\sigma_3=|z|^2$.  The resulting orbifold embeds into $\RR^3$ as the semi-algebraic variety $$C_k:=\{(s,t,u)\mid s^2+t^2=u^k,~u\geq 0\}.$$  This is the same differential subspace $C_k$ introduced in \hypref{x:varieties1}{Example}.
\eoe
\end{example}

\section{The Stratification of an Orbifold}\labell{s:strat}

In this section, we review stratified spaces from the perspective of subcartesian spaces.  For more details see Chapter 4 of  \cite{sniatycki} and \cite{LS}.  For a general introduction to stratified spaces, see \cite{pflaum}.  The main results of this section are \hypref{t:orbits}{Theorem} and \hypref{c:strat}{Corollary}.  The theorem states that the orbifold stratification is induced by the family of vector fields on the orbifold, which uses the theory of vector fields on subcartesian spaces developed by \'Sniatycki (see \cite{sniatycki}).  The corollary uses the fact that the family of vector fields of a subcartesian space is an invariant of the differential structure, and thus the orbifold stratification is an invariant of the orbifold differential structure.

\begin{definition}[Smooth Stratification]\labell{d:strat1}
Let $S$ be a subcartesian space.  Then a \emph{smooth stratification} of $S$ is a locally finite partition $\mathcal{M}$ of $S$ into locally closed and connected (embedded) submanifolds $M$, called the \emph{strata} of $\mathcal{M}$, which satisfy the following \emph{frontier condition}.

\begin{quote}
(\textbf{Frontier Condition:}) For any $M$ and $M'$ in $\mathcal{M}$, if $M'\cap\overline{M}\neq\emptyset$, then either $M=M'$ or $M'\subseteq\overline{M}\smallsetminus M$.
\end{quote}
\end{definition}

\begin{example}[Orbit-Type Stratification - Part I]\labell{x:ots}
Let $K$ be a Lie group acting properly on a manifold $M$.  Define for any closed subgroup $H$ of $K$ the subset of \emph{orbit-type} $(H)$ by $$M_{(H)}:=\{x\in M\mid \exists k\in K \text{ such that }\Stab_K(x)=kHk^{-1}\}.$$  Then the collection of all connected components of all (non-empty) subsets $M_{(H)}$ form a smooth stratification of $M$, called the \emph{orbit-type stratification} (see \cite[Theorem 2.7.4]{DK}).  Moreover, this stratification descends via the quotient map $\pi:M\to M/K$ to a smooth stratification on $M/K$, in which the strata are the connected components of $\pi(M_{(H)})$ as $H$ runs over closed subgroups of $K$ such that $M_{(H)}$ is non-empty (see \cite[Theorem 4.3.5]{sniatycki}).
\eoe
\end{example}

\begin{definition}[Orbifold Stratification]\labell{d:strat2}
Let $X$ be an orbifold.  Then $X$ admits a stratification, called the \emph{orbifold stratification} given locally as follows.  Let $(U,\Gamma,\phi)$ be a chart.  Then the orbit-type stratification on $U$ descends to a stratification on $U/\Gamma$ and hence on $\phi(U)$.
\end{definition}

\begin{lemma}[Orbifold Stratification is Well-Defined]\labell{l:strat}
Given an orbifold $X$, the orbifold stratification is independent of the charts of $X$; that is, it is well-defined.  Moreover, it is a smooth stratification in the sense of \hypref{d:strat1}{Definition}.
\end{lemma}

\begin{proof}
For any chart $(U,\Gamma,\phi)$ of $X$, the orbit-type stratification on $U$ descends to a smooth stratification on $\phi(U)$ (see \hypref{x:ots}{Example}).  Since the conditions of a stratification are local, we can construct a global stratification by piecing together the stratifications on each open set $\phi(U)$ for each chart $(U,\Gamma,\phi)$.  We only need to show that this stratification is independent of the chart.

To this end, let $n$ be the dimension of $X$.  Fix two charts $(U,\Gamma,\phi)$ and $(W,\Delta,\psi)$ such that there is an embedding $\lambda:(W,\Delta,\psi)\to(U,\Gamma,\phi)$.  We want to show that the strata of $\psi(W)$ match up with those of $\phi(U)$ via the inclusion $\psi(W)\subseteq\phi(U)$.  To accomplish this, it is enough to show that $\lambda$ induces a one-to-one correspondence between the orbit-type strata on $W$ and the connected components of the intersection of orbit-type strata of $U$ with $\lambda(W)$.

By \hypref{i:group monomorphism}{Item} of \hypref{r:orbifold}{Remark}, there is a group monomorphism $\bar{\lambda}:\Delta\to\Gamma$ such that $\lambda(\delta\cdot w)=\bar{\lambda}(\delta)\cdot \lambda(w)$ for all $\delta\in\Delta$, and for any $w\in W$ we have that $\bar{\lambda}$ induces a group isomorphism between $\Stab_\Delta(w)$ and $\Stab_\Gamma(\lambda(w))$.  It follows that $\lambda$ preserves orbit-types, and since $\lambda$ is continuous and continuous maps preserve connectedness, we have that $\lambda$ maps strata into connected components of the orbit-type strata of $U$ that intersect $\lambda(W)$.

Since $\lambda^{-1}:\lambda(W)\to W$ is also an embedding, we have that it maps strata of the $\bar{\lambda}(\Delta)$-action on $\lambda(W)$ into strata of $W$.  Let $u\in U_{(H)}\cap\lambda(W)$ with stabiliser $H\subseteq\Gamma.$  By \hypref{i:group monomorphism}{Item} of \hypref{r:orbifold}{Remark}, $H$ must be a subgroup of $\bar{\lambda}(\Delta)$, and it is the stabiliser of $u$ with respect to the action of $\bar{\lambda}(\Delta)$.  We conclude that $U_{(H)}\cap\lambda(W)=\lambda(W)_{(H)}$, and this completes the proof.

\end{proof}

\begin{remark}\labell{r:open dense stratum}
Given an orbifold $X$ with a chart $(U,\Gamma,\phi)$ in which $U$ is connected, there is an open, dense stratum of $\phi(U)$, the codimension-0 stratum.  The union of all of these yield an open and dense codimension-0 stratum of $X$, which is a manifold whose dimension equals the dimension of $X$.  Note that the dimension of $X$ is thus a topological invariant of it.  Indeed, the topological dimension at almost every $x\in X$ is equal to the dimension of $X$.
\end{remark}

\begin{definition}[Refinements and Minimality]\labell{d:refinement}
Let $S$ be a subcartesian space, and let $\mathcal{M}$ and $\mathcal{M}'$ be smooth stratifications on it.  Then $\mathcal{M}$ is said to \emph{refine} $\mathcal{M}'$ if for every $M\in\mathcal{M}$, there exists $M'\in\mathcal{M}'$ such that $M\subseteq M'$.  If $\mathcal{M}$ is not a refinement of any other smooth stratification on $S$, then we say that $\mathcal{M}$ is \emph{minimal}.
\end{definition}

\begin{example}[Orbit-Type Stratification - Part II]\labell{x:ots2}
Let $K$ be a non-trivial Lie group acting properly and effectively on a manifold $M$.  Then the orbit-type stratification on $M$ is not generally minimal (as the set of connected components of $M$ itself refines it).  On the other hand, the induced stratification on $M/K$ is minimal.  This is a result of Bierstone (see \cite{bierstone1}, \cite{bierstone2}).
\end{example}

\begin{definition}[Smooth Local Triviality]\labell{d:triviality}
Let $S$ be a subcartesian space, and let $\mathcal{M}$ be a smooth stratification on $S$.  Then $S$ is \emph{smoothly locally trivial} if for every $M\in\mathcal{M}$ and $x\in M$,
\begin{enumerate}
\item there is an open neighbourhood $U$ of $x$ such that the partition of $U$ into manifolds $N\cap U$ ($N\in\mathcal{M}$) yields a stratification of $U$,
\item there exists a subcartesian space $S'$ with smooth stratification $\mathcal{M}'$ which contains a singleton set $\{y\}\subseteq\mathcal{M}'$,
\item there exists a strata-preserving diffeomorphism $\varphi:U\to (M\cap U)\times S'$ sending $x$ to $(x,y)$.
\end{enumerate}
Note that the strata of $(M\cap U)\times S'$ are the sets $(M\cap U)\times M'$ where $M'\in\mathcal{M}'$.
\end{definition}

\begin{lemma}\labell{l:triviality}
Let $X$ be an orbifold.  Then the orbifold stratification on $X$ is smoothly locally trivial.
\end{lemma}

\begin{proof}
Since it is enough to prove this locally, we may focus on a chart $(U,\Gamma,\phi)$ of $X$.  By \hypref{r:isotropy group}{Remark}, we may assume th	at $U=\RR^n$, on which $\Gamma$ acts orthogonally.  Thus, we may apply the result Lemma 4.3.6 of \cite{sniatycki}.
\end{proof}

\begin{definition}[Tangent Bundles and Global Derivations]\labell{d:tangent}
Let $S$ be a subcartesian space.
\begin{enumerate}
\item Given a point $x\in S$, a \emph{derivation} of $\CIN(S)$ at $x$ is a linear map $v:\CIN(S)\to\RR$ that satisfies Leibniz' rule: for all $f,g\in\CIN(S)$, $$v(fg)=f(x)v(g)+g(x)v(f).$$  The set of all derivations of $\CIN(S)$ at $x$ forms a vector space, called the \emph{(Zariski) tangent space} of $x$, and is denoted $T_xS$. Define the \emph{(Zariski) tangent bundle} $TS$ to be the (disjoint) union $$TS:=\bigcup_{x\in S}T_xS.$$  Denote the canonical projection $TS\to S$ by $\tau$.
\item A \emph{(global) derivation} of $\CIN(S)$ is a linear map $Y:\CIN(S)\to\CIN(S)$ that satisfies Leibniz' rule: for any $f,g\in\CIN(S)$, $$Y(fg)=fY(g)+gY(f).$$  Denote the $\CIN(S)$-module of all derivations by $\Der\CIN(S)$.
\item Fix $Y\in\Der\CIN(S)$ and $x\in S$.  An \emph{integral curve} $\exp(\cdot Y)(x)$ of $Y$ through $x$ is a smooth map from a connected subset $I^Y_x\subseteq\RR$ containing 0 to $S$ such that $\exp(0Y)(x)=x$, and for all $f\in\CIN(S)$ and $t\in I^Y_x$ we have $$\frac{d}{dt}(f\circ\exp(tY)(x))=(Yf)(\exp(tY)(x)).$$  An integral curve is \emph{maximal} if $I_x^Y$ is maximal among the domains of all such curves.  We adopt the convention that the map $c:\{0\}\to S:0\mapsto c(0)$ is an integral curve of every global derivation of $\CIN(S)$.
\end{enumerate}
\end{definition}

\begin{remark}\labell{r:zariski}
\noindent
\begin{enumerate}
\item \labell{i:tau smooth} $TS$ is a subcartesian space with its differential structure generated by functions $f\circ\tau$ and $df$ where $f\in\CIN(S)$ and $d$ is the differential $df(v):=v(f)$.  The projection $\tau$ is smooth with respect to this differential structure (see \cite[page 4]{LSW} or \cite[Proposition 3.3.3]{sniatycki}).
\item \labell{i:str dim invt} Given $x\in S$, the dimension of $T_xS$ is invariant under diffeomorphism: if $\varphi:S\to R$ is a diffeomorphism of differential spaces, then $R$ is a subcartesian space, and the dimension of $T_{\varphi(x)}R$ is equal to that of $T_xS$.  Indeed, it is not hard to show that the pushforward $\varphi_*:TS\to TR$ sending $v\in T_xS$ to $\varphi_*v\in T_{\varphi(x)}R$ is a linear isomorphism on each tangent space.  (Recall that for any $f\in\CIN(R)$, we have $\varphi_*v(f)=v(f\circ\varphi).$)
\item Global derivations of $\CIN(S)$ are exactly the smooth sections of $\tau:TS\to S$ (see \cite[Proposition 3.3.5]{sniatycki}).
\item Let $Y\in\Der\CIN(S)$.  Then, for any $x\in S$, there exists a unique maximal integral curve $\exp(\cdot Y)(x)$ through $x$ (see \cite[Theorem 3.2.1]{sniatycki}).
\end{enumerate}
\end{remark}

\begin{definition}[Vector Fields and their Orbits]\labell{d:vector fields}
Let $S$ be a subcartesian space.
\begin{enumerate}
\item Let $D$ be a subset of $\RR\times S$ containing $\{0\}\times S$ such that $D\cap(\RR\times\{x\})$ is connected for each $x\in S$.  A map $\phi:D\to S$ is a \emph{local flow} if $D$ is open, $\phi(0,x)=x$ for each $x\in S$, and $\phi(t,\phi(s,x))=\phi(t+s,x)$ for all $x\in S$ and $s,t\in\RR$ for which both sides are defined.
\item A \emph{vector field} on $S$ is a derivation $Y$ of $\CIN(S)$ such that the map $(t,x)\mapsto\exp(tY)(x)$, sending $(t,x)$ to the maximal integral curve of $Y$ through $x$ evaluated at $t$, is a local flow.  Denote the set of all vector fields on $S$ by $\Vect(S)$.
\item Let $S$ be a subcartesian space, and let $\mathcal{M}$ be a smooth stratification of it.  Then the pair $(S,\mathcal{M})$ is said to \emph{admit local extensions of vector fields} if for any stratum $M\in\mathcal{M}$, any vector field $X_M$ on $M$, and any $x\in M$, there exist an open neighbourhood $U$ of $x$ and a vector field $X\in\Vect(S)$ such that $X_M|_{U\cap M}=X|_{U\cap M}.$
\item Let $S$ be a subcartesian space.  The \emph{orbit} of $\Vect(S)$ through a point $x$, denoted $O^S_x$, is the set of all points $y\in S$ such that there exist vector fields $Y_1,...,Y_k$ and real numbers $t_1,...,t_k\in\RR$ satisfying $$y=\exp(t_kY_k)\circ...\circ\exp(t_1Y_1)(x).$$
Denote by $\mathcal{O}_S$ the set of all orbits $\{O^S_x\mid x\in S\}$.

\end{enumerate}
\end{definition}

\begin{remark}\labell{r:vector fields}
Let $S$ be a subcartesian space.
\begin{enumerate}
\item \labell{i:vect invariant} Let $R$ be another subcartesian space, and let $F:R\to S$ be a diffeomorphism.  Then $F$ induces a bijection between $\Vect(R)$ and $\Vect(S)$.  Indeed, $F$ induces an isomorphism between the derivations of $\CIN(R)$ and those of $\CIN(S)$.  If $Z\in\Vect(R)$, then $F_*Z$ is a vector field on $S$: $$\frac{d}{dt}\Big|_{t=t_0}F\circ\exp(tZ)(x)=F_*Z|_{F(x)}.$$  The reverse direction also holds, and so the result follows.
\item \labell{i:vect module} $\Vect(S)$ is a $\CIN(S)$-module; that is, for any $f\in\CIN(S)$ and any vector field $Y\in\Vect(S)$, the derivation $fY$ is a vector field (see \cite[Corollary 4.71]{watts}).
\item \labell{i:loc ext vector fields} Let $\mathcal{M}$ be a smoothly locally trivial smooth stratification of $S$.  Then $(S,\mathcal{M})$ admits local extensions of vector fields (see \cite[Theorem 4.5]{LS} or \cite[Proposition 4.1.5]{sniatycki}).
\item \labell{i:LS1} Let $\mathcal{M}$ be a smooth stratification of $S$.  If $(S,\mathcal{M})$ admits local extensions of vector fields, then the set of orbits $\mathcal{O}_S$ forms a stratification of $S$, of which $\mathcal{M}$ is a refinement.  In particular, if $\mathcal{M}$ is minimal, then $\mathcal{M}=\mathcal{O}_S$ (see \cite[Theorem 4.6]{LS} or \cite[Theorem 4.1.6]{sniatycki}).
\item \labell{i:LS2} Let $\mathcal{O}_S$ be the set of orbits induced by $\Vect(S)$.  Then $\mathcal{O}_S$ is a stratification of $S$ if and only if it is locally finite and each $O\in\mathcal{O}_S$ is locally closed in $S$ (see  \cite[Theorem 4.3]{LS} or \cite[Corollary 4.1.3]{sniatycki}).
\end{enumerate}
\end{remark}

\begin{theorem}[The Orbifold Stratification is Induced by Vector Fields]\labell{t:orbits}
Let $X$ be an orbifold.  Then the orbifold stratification is given by the set of orbits $\mathcal{O}_X$ induced by $\Vect(X)$.
\end{theorem}

\begin{proof}
Let $(U,\Gamma,\phi)$ be a chart of $X$.  By \hypref{l:triviality}{Lemma}, the orbifold stratification on $\phi(U)$ is smoothly locally trivial.  Hence, it admits local extensions of vector fields by \hypref{i:loc ext vector fields}{Item} of \hypref{r:vector fields}{Remark}.  Thus, the orbits $\mathcal{O}_{\phi(U)}$ of $\Vect(\phi(U))$ form a stratification of $\phi(U)$ which is refined by the orbifold stratification on $\phi(U)$ by \hypref{i:LS1}{Item} of \hypref{r:vector fields}{Remark}.  However, since the orbifold stratification on $\phi(U)$ is minimal (see \hypref{x:ots2}{Example}), we conclude that the stratification by orbits $\mathcal{O}_{\phi(U)}$ is equal to the orbifold stratification.

By \hypref{l:strat}{Lemma}, we already know that the orbifold stratification is independent of chart.  Thus, it remains to show that for any $x\in\phi(U)$, we have that $O^{\phi(U)}_x$ is a connected component of $O^{X}_x\cap\phi(U)$.

We begin with the inclusion $O^{\phi(U)}_x\subseteq O^{X}_x\cap\phi(U)$.  Let $y\in O^{\phi(U)}_x$.  Then there exist vector fields $Y_1,...,Y_k\in\Vect(\phi(U))$ and $t_1,...,t_k\in\RR$ such that
\begin{equation}\labell{e:orbits}
y=\exp(t_kY_k)\circ...\circ\exp(t_1Y_1)(x).
\end{equation}
Fix $i\in\{1,...,k\}$.  The path $c:[0,t_i]\to\phi(U)$ defined by $$c:s\mapsto\exp(sY_i)\circ\exp(t_{i-1}Y_{i-1})\circ...\circ\exp(t_1Y_1)(x)$$ has a compact image in $\phi(U)$.  Now, $\phi(U)$ is open in $X$, and $X$ is normal (see \hypref{i:locally compact}{Item} of \hypref{r:orbifold}{Remark}).  So we can find an open neighbourhood $V$ of $c([0,t_i])$ and an open neighbourhood $W$ of the complement of $\phi(U)$ in $X$ that are disjoint.  Let $b_i:X\to\RR$ be a smooth bump function that is equal to $1$ on $V$ and has support in the complement of $W$ (it follows from \hypref{x:orbitspace}{Example} that $\phi(U)\subset\RR^N$ for some $N$, and so such a $b_i$ can be easily constructed).  Then by \hypref{i:vect module}{Item} of \hypref{r:vector fields}{Remark}, $b_iY_i\in\Vect(X)$.  Replacing each vector field $Y_i$ with $b_iY_i$ in \hypref{e:orbits}{Equation}, we obtain that $y\in O_x^{X}$.

Now consider the partition $P$ of $\phi(U)$ by connected components of $O\cap\phi(U)$ for each $O\in\mathcal{O}_X$.  Each element $Q$ of $P$ is an immersed submanifold of $\phi(U)$.  Moreover, each element $Q$ of $P$ is a finite union of strata of $\phi(U)$, and since each of these strata is locally closed, we have that $Q$ is locally closed.  Since for each $x\in\phi(U)$ we have $O^{\phi(U)}_x\subseteq O^X_x\cap\phi(U)$, we conclude that $P$ is locally finite.  It follows that $\mathcal{O}_X$ is locally finite and its elements are locally closed.  By \hypref{i:LS2}{Item} of \hypref{r:vector fields}{Remark}, $\mathcal{O}_X$ is a smooth stratification of $X$.  Moreover, $P$ is a smooth stratification of $\phi(U)$.  Since this stratification is refined by the orbifold stratification of $\phi(U)$, which is minimal, we conclude that $\mathcal{O}_{\phi(U)}=P$.
\end{proof}

\begin{corollary}[Invariance of Stratification]\labell{c:strat}
The orbifold stratification is an invariant of the orbifold differential structure.
\end{corollary}

\begin{proof}
This follows from \hypref{i:vect invariant}{Item} of \hypref{r:vector fields}{Remark} and \hypref{t:orbits}{Theorem}.
\end{proof}

\begin{example}[Reflections and Rotations in the Plane - Part III]\labell{x:varieties4}
Continuing \hypref{x:varieties3}{Example}, the strata of $\RR^2/D_k$ are given by the origin $\{(0,0)\}$, the two connected components of $\{(s,t)\mid t^2=s^{k},~s>0\}$, and the open dense stratum given by $\{(s,t)\mid t^2>s^{k},~s>0\}$.  Note that the codimension-1 and codimension-2 strata (called the \emph{singular strata}) together form the set $S_k$ of \hypref{x:varieties1}{Example}.

Similarly, the strata of $\RR^2/\ZZ_k$ are given by the origin $\{(0,0,0)\}$, and the open dense stratum $\{(s,t,u)\mid s^2+t^2=u^k,~u>0\}$.
\eoe
\end{example}

\section{Recovering the Charts}\labell{s:recover charts}

We begin with a discussion of orbifold covering spaces, based on \cite[Chapter 13]{thurston}; in particular, we need universal orbifold covering spaces for the proof of the \hyperref[t:main]{Main Theorem}.  Moreover, these motivate orbifold fundamental groups.  In \cite{HND} Haefliger and Ngoc Du show that the orbifold fundamental group can be obtained using the topology, stratification, and orders of points in codimension-2 strata (see \hypref{t:HND}{Theorem}).  In the previous sections we showed that the topology and stratifications are invariants of the orbifold differential structure, and in \hypref{p:milnor}{Proposition}, we show that the order of a point is also such an invariant.  This is important: while the order of a point in an orbifold may show up as the degree of an associated defining-polynomial (see \hypref{x:varieties5}{Example}), composition with a diffeomorphism may not yield a polynomial, and so ``degree'' does not make sense.  We then prove the \hyperref[t:main]{Main Theorem} at the end of the section.

\begin{definition}[Orbifold Covering Space]\labell{d:covering space}
Let $X$ be an orbifold, and fix a base point $x_0$ in the codimension-0 stratum of $X$.
\begin{enumerate}
\item An \emph{orbifold covering space} of $X$ is an orbifold $\tilde{X}$ with a base point $\tilde{x}_0$ in its codimension-0 stratum, and an orbifold smooth ``projection'' map $p:\tilde{X}\to X$ which sends $\tilde{x}_0$ to $x_0$.  For any $x\in X$ we require that there is a chart $(U,\Gamma,\phi)$ of $X$ with $x\in\phi(U)$ and for each connected component $C_i$ of $p^{-1}(\phi(U))$ there is a $\Gamma$-equivariant diffeomorphism $\Psi_i:C_i\to U/\Gamma_i$ where $\Gamma_i\subseteq\Gamma$ is a subgroup.
\item $X$ is called a \emph{good orbifold} if there exists an orbifold covering space that is a smooth manifold; otherwise, it is called a \emph{bad orbifold}.
\item A \emph{universal orbifold covering space} of $X$ is a connected orbifold covering space $p:\tilde{X}\to X$ such that
if $\tilde{X}'$ is any other orbifold covering space of $X$ with projection $p':\tilde{X}'\to X$, then there is a lifting of $p$ via $p'$ to a map $q:\tilde{X}\to\tilde{X}'$ by which $\tilde{X}$ is an orbifold covering space of $\tilde{X}'$.
\item If $p:\tilde{X}\to X$ is a universal orbifold covering space of $X$ with base point $\tilde{x}_0\in p^{-1}(x_0)$, then for any other $y\in p^{-1}(x_0)$, there is a deck transformation taking $\tilde{x}_0$ to $y$; that is, an orbifold diffeomorphism $f:\tilde{X}\to\tilde{X}$ such that $p\circ f=p$ and $f(\tilde{x}_0)=y$.  The group of deck transformations of $\tilde{X}$ is called the \emph{orbifold fundamental group} of $X$.  (See \cite[Definition 13.2.5]{thurston}.)
\end{enumerate}
\end{definition}

\begin{remark}\labell{r:covering space}
\noindent
\begin{enumerate}
\item \labell{i:not top} Note that an orbifold covering space of an orbifold $X$ in general is not a covering space in the topological sense.
\item \labell{i:simply connected} If $X=M/\Gamma$ where $M$ is a simply connected manifold on which a finite group $\Gamma$ acts, then $M$ is the universal orbifold covering space of $X$.  If $M$ is not simply connected, then we can take its (topological) universal covering space as the universal orbifold covering space of $X$.
\item \labell{i:univ cover exists} Let $X$ be an orbifold.  Then $X$ has a universal orbifold covering space $\tilde{X}$, which is unique up to orbifold diffeomorphism.  Moreover, if $X$ is a good orbifold, then $\tilde{X}$ is a simply connected manifold.  (see \cite[Proposition 13.2.4]{thurston}.)
\end{enumerate}
\end{remark}

\begin{definition}[Order of a Point]\labell{d:point order}
Let $X$ be an orbifold and let $x\in X$ with isotropy group $\Gamma_x$.  Then, the \emph{order} of $x$ is equal to the order of the group $\Gamma_x$.
\end{definition}

\begin{remark}\labell{r:point order}
It follows from \hypref{r:isotropy group}{Remark} that the order of a point is well-defined.
\end{remark}

\begin{theorem}[Recovering the Groups]\labell{t:HND}
Let $X$ be a connected orbifold.  Then a presentation for the orbifold fundamental group can be constructed using the topology, stratification, and the orders of points in codimension-2 strata.
\end{theorem}

\begin{proof}
This is proved by Haefliger and Ngoc Du in \cite{HND}.  See also Section 1.3 of \cite{davis}.  We briefly explain the algorithm here.  Let $X_{reg}$ be the differential subspace of $X$ consisting of codimension-0 and codimension-1 strata.  Fix a base point $x$ in the codimension-0 stratum. Let $G$ be the (topological) fundamental group of $X_{reg}$ with respect to $x$.
\begin{enumerate}
\item For each codimension-1 stratum $S_i$, and for each homotopy class $\mu$ of paths starting at $x$ and ending on $S_i$ attach a generator $\beta_{i,\mu}$ to $G$ with relation $\beta_{i,\mu}^2=1$.
\item For each codimension-2 stratum $R$ in the closure of a codimension-1 stratum, for each pair of codimension-1 strata $S_i$, $S_{i'}$ with $R$ in their closures, and for each pair $\beta_{i,\mu}$, $\beta_{i',\mu'}$ (where $\mu\neq\mu'$) as constructed in Item 1 above, add the relation $(\beta_{i,\mu}\beta_{i',\mu'})^k=1$ where $2k$ is the order of any point in $R$.
\item For each codimension-2 stratum $T_j$ not in the closure of a codimension-1 stratum, let $\alpha_j$ be an element of $G$ represented by a loop starting at $x$ and going around $T_j$.  Then add the relation $\alpha_j^k=1$ to $G$ where $k$ is the order of any point in $T_j$.
\end{enumerate}
The resulting group is the orbifold fundamental group of $X$.
\end{proof}

\begin{example}[Reflections and Rotations in the Plane - Part IV]\labell{x:varieties5}
Consider the orbifold $\RR^2/D_k$.  In \hypref{x:varieties3}{Example} we saw that $\RR^2/D_k$ embedded into $\RR^2$ as the semi-algebraic variety $$R_k:=\{(s,t)\mid t^2\leq s^{k},~s\geq0\},$$ with its strata listed in \hypref{x:varieties4}{Example}.  Applying the algorithm in the proof of \hypref{t:HND}{Theorem}, we have that the orbifold fundamental group is $$\langle\beta_1,\beta_2\mid \beta_1^2=\beta_2^2=(\beta_1\beta_2)^k=1\rangle.$$  But this is exactly $D_k$.

Similarly, consider the orbifold $\RR^2/\ZZ_k$.  In \hypref{x:varieties3}{Example} we saw that $\RR^2/\ZZ_k$ embedded into $\RR^3$ as the semi-algebraic variety $$C_k:=\{(s,t,u)\mid s^2+t^2=u^k,~u\geq 0\},$$ with its strata listed in \hypref{x:varieties4}{Example}.  Applying the algorithm in the proof of \hypref{t:HND}{Theorem}, we have that the orbifold fundamental group is $$\langle\alpha\mid \alpha^k=1\rangle.$$  This is exactly $\ZZ_k$.
\eoe
\end{example}

\begin{definition}[Codimension of a Germ]\labell{d:milnor}
Let $\mathcal{E}_n$ be the $\RR$-algebra of germs of smooth real-valued functions at $0\in\RR^n$: $$\mathcal{E}_n=\CIN(\RR^n)/\!\sim$$ where $f\sim g$ if there exists an open neighbourhood of $0$ on which $f=g$.  (In practice, where it doesn't cause confusion, we will often identify an element of $\mathcal{E}_n$ with one of its representatives.)  Let $f\in\mathcal{E}_n$, and define $J_f$ to be the \emph{Jacobian ideal} of $f$, which is the ideal of $\mathcal{E}_n$ generated by the germs of partial derivatives of $f$ at $0$: $$J_f=\left\langle\frac{\partial f}{\partial x_1},...,\frac{\partial f}{\partial x_n}\right\rangle.$$  The \emph{codimension} of (the germ of) $f$ at $0$, denoted $\cod(f)$, is defined to be the dimension of the quotient algebra $\mathcal{E}_n/J_f$ as a vector space.
\end{definition}

\begin{proposition}[Codimension of a Germ is an Invariant]\labell{p:milnor}
Let $f\in\CIN(\RR^n)$, with $f(0)=a$.  Assume that $0$ is a critical point of $f$.  Then the codimension of (the germ of) $f$ at $0$ is invariant under diffeomorphism.  In particular, $\cod(f)$ is an invariant of the differential structure on the differential subspace $f^{-1}(a)\subseteq\RR^n$.
\end{proposition}

\begin{proof}
The proof that the $\cod(f)$ is invariant under diffeomorphism is an immediate consequence of the chain rule.  See \cite[Theorem 2.12]{gibson} for more details.

Next, let $\varphi$ be a diffeomorphism between $f^{-1}(a)$ and a differential space $(S,\CIN(S))$.  Then $(S,\CIN(S))$ is subcartesian.  Let $x=\varphi(0)\in S$.  Then there is an open neighbourhood $U$ of $x$ in $S$ and a diffeomorphism $\psi:U\to\tilde{U}$ where $\tilde{U}$ is a differential subspace of $\RR^m$.  Without loss of generality, we may choose $m$ to be the minimal such integer for which the diffeomorphism $\psi$ exists.  By \cite[Lemma 3.4]{LSW} this is equal to the dimension of $T_xS$, which is invariant under diffeomorphism by \hypref{i:str dim invt}{Item} of \hypref{r:zariski}{Remark}.  Thus $n\geq m$.  If $n>m$, then embed $\tilde{U}\subseteq\RR^m$ into $\RR^n$ by $$(x_1,...,x_m)\mapsto(x_1,...,x_m,0,...,0).$$  In either case we now have a diffeomorphism $\tilde{\varphi}$ from $f^{-1}(a)$ to $\tilde{U}$ which are both differential subspaces of $\RR^n$.  Without loss of generality, assume that $\tilde{\varphi}(0)=0$.  By \cite[Theorem 6.3]{watts}, $\tilde{\varphi}$ extends to a diffeomorphism from an open neighbourhood of $0\in\RR^n$ to itself.  The result now follows.
\end{proof}

\begin{example}[Reflections and Rotations in the Plane - Part V]\labell{x:varieties6}
Continuing \hypref{x:varieties5}{Example}, recall that the singular strata of the orbit space $\RR^2/D_k$ are given by the relations
\begin{align*}
t^2-s^{k}=&~0,\\
s\geq&~0.
\end{align*}
The codimension of $f(s,t)=t^2-s^{k}$ is computed as follows.  The partial derivatives are $$\frac{\partial f}{\partial s}(s,t)=-ks^{k-1} \text{ and } \frac{\partial f}{\partial t}(s,t)=2t.$$  It follows that $\mathcal{E}_2/J_f$ is generated by $$s,s^2,...,s^{k-2}$$ and so $\cod(f)=k-2+1=k-1$ (where we add one to account for the constant functions). Note that $|D_k|=2(\cod(f)+1)$.

Similarly, recall that $\RR^2/\ZZ_k$ is given by the relations
\begin{align*}
s^2+t^2-u^k=&~0\\
u\geq&~0
\end{align*}
The codimension of $f(s,t,u)=s^2+t^2-u^k$ is computed as follows.  The partial derivatives are $$\frac{\partial f}{\partial s}(s,t,u)=2s,~\frac{\partial f}{\partial t}(s,t,u)=2t, \text{ and } \frac{\partial f}{\partial u}(s,t,u)=-ku^{k-1}.$$  It follows that $\mathcal{E}_3/J_f$ is generated by $$u,u^2,...,u^{k-2}$$ and so $\cod(f)=k-2+1=k-1$.  Note that $|\ZZ_k|=\cod(f)+1$.
\eoe
\end{example}

\begin{theorem}[Order of a Point is an Invariant]\labell{t:point order}
Let $X$ be an orbifold, and let $x\in X$.  If $x$ is in a codimension-2 stratum, then the order of $x$ is an invariant of the orbifold differential structure.  Consequently, the order of any point of $X$ is an invariant of the orbifold differential structure.
\end{theorem}

\begin{proof}
Recall that the orbifold stratification is an invariant of the differential structure by \hypref{c:strat}{Corollary}.  Let $n$ be the dimension of $X$, let $\Gamma_x$ be an isotropy group of $x$, and let $M$ be the stratum containing $x$.  By \hypref{l:triviality}{Lemma} there is an open neighbourhood $U$ of $x$, a smooth stratified subcartesian space $S'$ with a one-point stratum $\{y\}$, and a strata-preserving diffeomorphism $U\to (M\cap U)\times S'$ sending $x$ to $(x,y)$.  Let $(\RR^n,\Gamma_x,\phi)$ be a chart at $x$ such that $\phi(0)=x$, in which $\Gamma_x$ acts orthogonally.  Without loss of generality, assume that $U=\phi(\RR^n)$.  Let $E=(\RR^n)^{\Gamma_x}$ be the linear subspace of $\Gamma_x$-fixed points, and $F$ be an orthogonal complement to $E$.  Then since $\phi(0)=x$ and $0$ is a fixed point, we have that $\phi(E)=M\cap U$.  Since $\Gamma_x$ acts trivially on $E$, we have that $E\cong\RR^{n-2}$, and so $F\cong\RR^2$, on which $\Gamma_x$ acts with unique fixed point $0$.  Hence, by \hypref{t:LDV}{Theorem}, $\Gamma_x$ is a dihedral group (if $M$ is in the closure of a codimension-1 stratum) or a cyclic group (if $M$ is not in the closure of a codimension-1 stratum).  By \hypref{x:varieties6}{Example} and \hypref{p:milnor}{Proposition}, the order of $\Gamma_x$ can be obtained from invariants of the orbifold differential structure.

For the second statement, recall that the orbifold fundamental group of any orbifold can be obtained from the topology, stratification, and orders of points of codimension-2 strata by \hypref{t:HND}{Theorem}.  By \hypref{i:topology}{Item} of \hypref{p:diff orb}{Proposition}, \hypref{c:strat}{Corollary}, and what was proved above, we have that the orbifold fundamental group can be obtained from invariants of the orbifold differential structure.  From \hypref{i:simply connected}{Item} of \hypref{r:covering space}{Remark}, if $(\RR^n,\Gamma_x,\phi)$ is a chart of an orbifold $X$ in which $x=\phi(0)$ and $\Gamma_x$ is its isotropy group (which always exists by \hypref{r:isotropy group}{Remark}), then the orbifold fundamental group of $\phi(U)$ is isomorphic to $\Gamma_x$.  Thus, $|\Gamma_x|$ can be obtained from invariants of the orbifold differential structure.
\end{proof}

\begin{definition}[Link at a Point]\labell{d:link}
Let $X$ be an orbifold of dimension $n$, and let $x\in X$.  Let $\Gamma_x$ be an isotropy group of $x$ with associated chart $(\RR^n,\Gamma_x,\phi)$.  Without loss of generality, assume that $\Gamma_x\subset O(n)$.  Then $\SS^{n-1}$ is a $\Gamma_x$-invariant set.  Define the \emph{link} at $x$ to be the image $S:=\phi(\SS^{n-1})$.
\end{definition}

\begin{lemma}\labell{l:local normal form}
Let $X$ be an orbifold of dimension $n$ and let $x\in X$ with isotropy group $\Gamma_x$.  Let $(\RR^n,\Gamma_x,\phi)$ be a chart with $\phi(0)=x$ and such that $\Gamma_x\subset O(n)$.  Then there is a diffeomorphism $\Phi_S$ from the link $S$ at $x$ as a differential subspace of $X$ to $\SS^{n-1}/\Gamma_x$, where the action of $\Gamma_x$ on $\SS^{n-1}$ is the restriction of that on $\RR^n$ in the chart.  Moreover, $S$ has a smooth stratification given by the connected components of $S\cap M$ where $M$ runs over strata of $X$, and $\Phi_S$ preserves this with respect to the orbifold stratification on $\SS^{n-1}/\Gamma_x$.  Finally, $\Phi_S$ preserves the orders of points contained in the codimension-2 strata of $S$.
\end{lemma}

\begin{proof}
The existence of the diffeomorphism $\Phi_S$ follows from the definition of a chart and the following commutative diagram.
$$\xymatrix{
\SS^{n-1} \ar[r] \ar[d]_{\phi|_{\SS^{n-1}}} & \RR^n \ar[d]^{\phi} \\
S \ar[r] & \phi(\RR^n)\\
}$$
Let $M$ be a stratum of $X$, and let $y\in C\subseteq S\cap M$ where $C$ is a connected component of $S\cap M$.  Then there exists a subgroup $H$ of $\Gamma$ such that $y\in\phi(\RR^n_{(H)})$.  Note that $y\neq\phi(0)$.  Also, $\RR^n_{(H)}$ is a cone; that is, it is closed under scalar multiplication by non-zero real numbers.  Thus, $y\in\phi(\SS^{n-1}_{(H)})$, and since $y\in C$ is arbitrary, we have $C\subseteq\phi(\SS^{n-1}_{(H)})$.  For the opposite inclusion, fix $y\in\SS^{n-1}/\Gamma_x$ and let $H$ be a subgroup of $\Gamma_x$ such that $y\in\phi(\SS^{n-1}_{(H)})$.  Then, similar to the previous argument, $y\in\phi(\RR^n_{(H)})$.  Thus, there is a stratum $M$ of the orbifold stratification on $X$ such that if $C$ is the connected component of the stratum $\phi(\SS^{n-1}_{(H)})$ containing $y$, then $C\subseteq S\cap M$.  Finally, the fact that $\Phi_S$ preserves the orders of points follows immediately from the definitions.
\end{proof}

\begin{proof}[Proof of Main Theorem]
First, recall that the dimension of $X$ is a topological invariant (see \hypref{r:open dense stratum}{Remark}).  Moreover, this topology, the orbifold stratification, and the order of points in codimension-2 strata are all invariants of the differential structure $\CIN(X)$ by \hypref{i:topology}{Item} of \hypref{p:diff orb}{Proposition}, \hypref{c:strat}{Corollary}, and \hypref{t:point order}{Theorem} respectively.

If the dimension of $X$ is $0$, then $X$ is a countable set of points with the discrete topology, and the orbifold atlas is trivial.

Now, assume that $X$ has dimension $1$.  Then there are no codimension-2 strata, and applying \hypref{t:HND}{Theorem} to $X$ yields the following isotropy groups $\Gamma_x$ at each point $x\in X$:
\begin{enumerate}
\item If $x$ is in the open dense stratum, then $\Gamma_x=\{1\}$.
\item If $x$ is a codimension-1 stratum, then $\Gamma_x=\ZZ_2$.
\end{enumerate}

Thus, we can construct the following charts.

\begin{enumerate}
\item If $x$ is in a codimension-0 stratum, then there is an open neighbourhood $U$ of $x$ such that $U\cong\RR\cong\RR/\{1\}$; that is, $U$ is diffeomorphic to a connected open interval of $\RR$.  We thus take a chart $(\RR,\{1\},\phi)$ where $\phi$ is the diffeomorphism from $\RR$ onto $U$.
\item If $x$ is equal to a codimension-1 stratum, then there is only one non-trivial action of $\ZZ_2$ on $\RR$ given by $\pm1\cdot u=\pm u$.  So we must take as a chart near $x$ the triple $(U,\Gamma_x,\phi)=(\RR,\ZZ_2,\phi)$ where $\phi:\RR\to\phi(U)$ is the quotient map of this $\ZZ_2$-action.
\end{enumerate}

This completes the one-dimensional case.

Next, assume that the dimension of $X$ is $2$.  Applying \hypref{t:HND}{Theorem} to $X$ yields the following four possible isotropy groups $\Gamma_x$ at each point $x\in X$:

\begin{enumerate}
\item If $x$ is in a codimension-0 stratum, then $\Gamma_x=\{1\}$.
\item If $x$ is in a codimension-1 stratum, then $\Gamma_x=\ZZ_2$.
\item If $x$ is equal to a codimension-2 stratum that is in the closure of a codimension-1 stratum, and the order of $x$ is $2k$, then $\Gamma_x=D_k$.
\item If $x$ is equal to a codimension-2 stratum that is not in the closure of a codimension-1 stratum, and the order of $x$ is $k$, then $\Gamma_x=\ZZ_k$.
\end{enumerate}

We construct the following charts.

\begin{enumerate}
\item If $x$ is in a codimension-0 stratum, then there is an open neighbourhood $U$ of $x$ such that $U\cong\RR^2\cong\RR^2/\{1\}$.  Similar to what we did for the 1-dimensional case, take $(\RR^2,\{1\},\phi)$ to be a chart.
\item If $x$ is in a codimension-1 stratum, then there is an open neighbourhood $U$ of $x$ such that $U\cong\RR^2/\ZZ_2$ where $\ZZ_2$ acts by reflection through some line passing through the origin.
\item If $x$ is equal to a codimension-2 stratum that is in the closure of a codimension-1 stratum, and the order of $x$ is $2k$, then we take as a chart near $x$ the triple $(\RR^2,D_k,\phi)$ where $D_k$ acts on $\RR_2\cong\CC$ by reflections (see \hypref{x:varieties2}{Example}).
\item If $x$ is equal to a codimension-2 stratum that is not in the closure of a codimension-1 stratum, and the order of $x$ is $k$, then we take as a chart near $x$ the triple $(\RR^2,\ZZ_k,\phi)$ where $\ZZ_k$ acts on $\RR^2\cong\CC$ by rotations (see \hypref{x:varieties2}{Example}).
\end{enumerate}

By \hypref{t:LDV}{Theorem} and \hypref{r:LDV}{Remark}, this exhausts all the possible scenarios in the 2-dimensional case.

Now, as an induction hypothesis, assume that we can reconstruct an atlas for any orbifold of dimension $n$.  Let $X$ be an orbifold of dimension $n+1$.  Fix $x\in X$ and let $\Gamma_x$ be its isotropy group at $x$.  Our goal is to reconstruct a chart $(\RR^{n+1},\Gamma_x,\phi)$ about $x$ such that $\phi(0)=x$.  Let $S$ be the link at $x$.  By \hypref{l:local normal form}{Lemma}, there is a strata-preserving diffeomorphism that preserves the order of points on codimension-2 strata from $S$ to $\SS^n/\Gamma_x$ for some action of $\Gamma_x$ on $\SS^n$.  By our induction hypothesis, we now have enough information on $S$ to obtain an orbifold atlas on $S$.

Now, $S$ is a good orbifold.  Thus, by \hypref{i:univ cover exists}{Item} of \hypref{r:covering space}{Remark}, there is a simply-connected manifold that serves as a universal orbifold covering space for $S$, and this is unique up to equivariant diffeomorphism.  Hence (safely assuming that $n\geq 2$), $\SS^n$ is the universal orbifold covering space for $S$, with the action of $\Gamma_x$ given by deck transformations.  Extend this action to the unique orthogonal action of $\Gamma_x$ on $\RR^{n+1}$ such that
$$\gamma\cdot x:=\begin{cases}
0 & \text{ if $x=0$,}\\
|x|(\gamma\cdot\frac{x}{|x|}) & \text{ if $x\neq 0$.}\\
\end{cases}$$
This finishes the construction of the chart $(\RR^{n+1},\Gamma_x,\phi)$.  Since $x\in X$ is arbitrary, we are done.
\end{proof}

\section{Proof of Theorem A}\labell{s:functor}

The purpose of this section is to express the \hyperref[t:main]{Main Theorem} in terms of a functor.  We choose to use the weak 2-category of effective proper \'etale Lie groupoids with bibundles as arrows and isomorphisms of bibundles as 2-arrows.  We give the definition of these objects, arrows, and 2-arrows, but the reader should consult \cite{lerman} for a more detailed exposition.  Similar categories have been developed, toward which the \hyperref[t:main]{Main Theorem} could be tailored, but we do not do this here.  These categories include that of Hilsum-Skandalis \cite{HS}, as well as the calculus of fractions developed by Pronk \cite{pronk}.

Set $G=(G_1\toto G_0)$ as our notation for a Lie groupoid, with $s:G_1\to G_0$ and $t:G_1\to G_0$ the source and target maps, respectively.

\begin{definition}[Effective Proper \'Etale Lie Groupoid]\labell{d:groupoid}
\noindent
\begin{enumerate}
\item Let $G$ be a Lie groupoid.  Then $G$ is \emph{\'etale} if the source and target maps are local diffeomorphisms.

\item Let $M$ be a manifold.  Denote by $\Gamma(M)$ the topological groupoid with objects the set of points of $M$, and arrows the space of germs of (local) diffeomorphisms equipped with the sheaf topology.  This is an \'etale groupoid in the topological sense; \emph{i.e.} the source and target maps are local homeomorphisms.  It attains a smooth structure via these local homeomorphisms.

\item Let $G$ be an \'etale Lie groupoid, and let $\Gamma(G_0)$ be the groupoid of germs associated to $G_0$.  Then for each arrow $(g:x\to y)\in G_1$ there exist an open neighbourhood $U$ of $g$ and a diffeomorphism $$\phi_g=t|_U\circ (s|_U)^{-1}.$$  The germ of $\phi_g$ is an element of $\Gamma(G_0)$, and we have a smooth functor $\gamma:G\to\Gamma(G_0)$, sending $g$ to $\phi_g$, and objects to themselves.  $G$ is \emph{effective} if $\gamma$ is faithful.

\item A Lie groupoid $G$ is \emph{proper} if the smooth functor $(s,t):G_1\to G_0\times G_0$ is a proper map between manifolds.
\end{enumerate}
\end{definition}

\begin{remark}\labell{r:groupoid}
Any effective proper \'etale Lie groupoid is Morita equivalent to the effective groupoid associated to an orbifold constructed using pseudogroups.  This construction yields a one-to-one correspondence between orbifolds in the classical sense, and Morita equivalence classes of effective proper \'etale Lie groupoids.  See \cite{MM} for definitions, details, and a proof (\cite[Theorem 5.32]{MM}).  The important point for our purposes is that given one of these groupoids $G_1\toto G_0$, the orbit space $G_0/G_1$ is the underlying set of the orbifold.
\end{remark}

\begin{definition}[Bibundle]\labell{d:bibundle}
\noindent
\begin{enumerate}
\item
Let $G=(G_1\toto G_0)$ and $H=(H_1\toto H_0)$ be Lie groupoids. Then a \emph{bibundle} $P:G\to H$ is a manifold $P$ equipped with a left groupoid action of $G$ with anchor map $a_L:P\to G_0$, and a right groupoid action of $H$ with anchor map $a_R:P\to H_0$ such that the following are satisfied.
\begin{enumerate}
\item The two actions commute.
\item $a_L:P\to G_0$ is a principal (right) $H$-bundle.
\item $a_R$ is $G$-invariant.
\end{enumerate}

$$\xymatrix{
G_1 \ar@<0.5ex>[d] \ar@<-0.5ex>[d] & P \ar[dl]_{a_L} \ar[dr]^{a_R} & H_1 \ar@<0.5ex>[d] \ar@<-0.5ex>[d] \\
G_0 & & H_0 \\
}$$

\item Let $G$ and $H$ be Lie groupoids, and let $P:G\to H$ and $Q:G\to H$ be bibundles between them.  An \emph{isomorphism of bibundles} $\alpha:P\to Q$ is a diffeomorphism that is ($G$-$H$)-equivariant; that is, $\alpha(h\cdot p\cdot g)=h\cdot \alpha(p)\cdot g.$

\item Let $G=(G_1\toto G_0)$ and $H=(H_1\toto H_0)$ be Lie groupoids, and let $P:G\to H$ be a bibundle between them.  $P$ is \emph{invertible} if its right anchor map $a_R:P\to H_0$ makes $P$ into a principal (left) $G$-bundle, defined similarly to a principal (right) bundle.  In this case, we can construct a bibundle $P^{-1}:H\to G$ by switching the anchor maps, inverting the left $G$-action into a right $G$-action, and doing the opposite for the $H$-action.  Then, $P\circ P^{-1}$ is isomorphic to the bibundle corresponding to the identity map on $H$, and $P^{-1}\circ P$ isomorphic to the bibundle representing the identity map on $G$.  In the case that $G$ and $H$ admit an invertible bibundle between them, they are called \emph{Morita equivalent} groupoids.
\end{enumerate}
\end{definition}

\begin{definition}[Weak 2-Category $\orb$]\labell{d:orb}
Lie groupoids with bibundles as arrows and isomorphisms of bibundles as 2-arrows form a weak 2-category.  See \cite{lerman} for more details.  Effective proper \'etale Lie groupoids form a full (weak) subcategory, and will be denoted by $\orb$.  Many view this (or slight modifications to this definition) to be ``the'' category of effective orbifolds.
\end{definition}

\begin{lemma}\labell{l:arrows}
Let $G_1\toto G_0$ and $H_1\toto H_0$ be two effective proper \'etale Lie groupoids, and let $P$ be a bibundle between them.  Then $P$ descends to a unique smooth map $\bar{P}:G_0/G_1\to H_0/H_1$ such that the following diagram commutes.

$$\xymatrix{
G_1 \ar@<0.5ex>[d] \ar@<-0.5ex>[d] & P \ar[dl]_{a_L} \ar[dr]^{a_R} & H_1 \ar@<0.5ex>[d] \ar@<-0.5ex>[d] \\
G_0 \ar[d]_{\pi_G} & & H_0 \ar[d]^{\pi_H} \\
G_0/G_1 \ar[rr]_{\bar{P}} & & H_0/H_1
}$$

Moreover, if $P$ is a Morita equivalence, then $\bar{P}$ is a diffeomorphism.  Finally, if $Q$ is another bibundle between $G_1\toto G_0$ and $H_1\toto H_0$ that is isomorphic to $P$, then $\bar{P}=\bar{Q}$.
\end{lemma}

\begin{proof}
Fix $x\in G_0$ and denote by $[x]$ the point $\pi_G(x)$.  Then define $$\bar{P}([x]):=\pi_H\circ a_R\circ\sigma(x)$$ where $\sigma$ is a smooth local section of $a_L$ about $x$.  Such a local section exists since $a_L$ is a surjective submersion, by definition of a principal $H$-bundle.

We claim that $\bar{P}$ is independent of the local section chosen, as well as the representative $x$.  Indeed, let $y\in G_0$ be another representative of $[x]$.  Then there exists $g\in G_1$ such that $s(g)=x$ and $t(g)=y$.  So, $a_L(g\cdot\sigma(x))=y$, and hence $g\cdot\sigma(x)\in a_L^{-1}(y)$.  Let $\sigma'$ be a local section of $a_L$ about $y$. Since $a_L:P\to G_0$ is a principal $H$-bundle, there exists $h\in H_1$ such that $(g\cdot\sigma(x))\cdot h=\sigma'(y)$.  Since the $G$- and $H$-actions on $P$ commute and $a_R$ is $G$-invariant, it follows that $a_R(\sigma'(y))=s(h)$. Since $a_R(\sigma(x))=t(h)$, we have $$\pi_H(a_R(\sigma(x)))=\pi_H(a_R(\sigma'(y))).$$

To show uniqueness, consider $p\in P$.  In order for the diagram above to commute, we require that $\pi_G(a_L(p))$ be sent to $\pi_H(a_R(p))$.  This defines a unique map, which is equal to $\bar{P}$.

Denote the quotient differential structures on $G_0/G_1$ and $H_0/H_1$ by $\CIN(G_0/G_1)$ and $\CIN(H_0/H_1)$, respectively.  Denote by $\CIN(G_0)^{G_1}$ and $\CIN(H_0)^{H_1}$ the spaces of smooth invariant functions on $G_0$ and $H_0$, respectively.  Fix $f\in\CIN(H_0/H_1)$.  Then, there exists $\tilde{f}\in\CIN(H_0)^{H_1}$ such that $\tilde{f}=\pi_H^*f$.  By definition of a right $H$-action, $a_R^*\tilde{f}$ is $H$-invariant on $P$.  Since $a_L:P\to G_0$ is a principal $H$-bundle, $a_R^*\tilde{f}$ descends to a smooth function $\tilde{f}'\in\CIN(G_0)$: $$a_L^*\tilde{f}'=a_R^*\tilde{f}.$$  By definition of a left $G$-action, and using the fact that $a_R$ is $G$-invariant, we obtain that $\tilde{f}'\in\CIN(G_0)^{G_1}$.  Therefore, $\tilde{f}'$ descends to a function $f'\in\CIN(G_0/G_1)$, and $f'=\bar{P}^*f$.

Next, $P$ is a Morita equivalence if and only if $P$ is invertible; that is, $a_R:P\to H_0$ is a principal $G$-bundle.  It follows immediately that in this case, $\bar{P}$ is a diffeomorphism.

Finally, the fact that isomorphic bibundles $P$ and $Q$ descend to the same smooth map $\bar{P}=\bar{Q}$ between orbit spaces comes immediately from the uniqueness of $\bar{P}$ and the fact that $\alpha$ is ($G$-$H$)-equivariant.
\end{proof}

\begin{proof}[Proof of Theorem A]
We define a functor $F:\orb\to\diffsp$ as follows.  Let $G_1\toto G_0$ be an effective proper \'etale Lie groupoid.  Then $F(G_1\toto G_0)$ is the orbit space $G_0/G_1$ equipped with the quotient differential structure.  Let $G_1\toto G_0$ and $H_1\toto H_0$ be two effective proper \'etale Lie groupoids, and let $P$ be a bibundle between them.  Then define $F(P)$ to be $\bar{P}$ as defined in \hypref{l:arrows}{Lemma}.  $F$ trivialises 2-arrows by \hypref{l:arrows}{Lemma}.

To show that $F$ is a functor, note that if $P:(G_1\toto G_0)\to(G_1\toto G_0)$ is the identity bibundle, then $\bar{P}$ is the identity map on $G_0/G_1$.  We also need to show that $F$ respects composition.  Let $G_1\toto G_0$, $H_1\toto H_0$, and $K_1\toto K_0$ be effective proper \'etale Lie groupoids, and let $P:(G_1\toto G_0)\to(H_1\toto H_0)$ and $Q:(H_1\toto H_0)\to(K_1\toto K_0)$ be bibundles.  The composition of $P$ and $Q$ is defined to be the quotient $Q\circ P:=(P\times_{H_0} Q)/H_1$ where $P\times_{H_0} Q$ is the fibred product with respect to anchor maps $a^P_R:P\to H_0$ and $a^Q_L:Q\to H_0$, on which $H_1\toto H_0$ acts diagonally.  Note that $F(Q\circ P)=\overline{Q\circ P}$ is the unique map making the following diagram commute.

$$\xymatrix{
 & P\times_{H_0}Q \ar[dl]_{\pr_1} \ar[dr]^{\pr_2} & \\
P \ar[d]_{\pi_G\circ a_L^P} & & Q \ar[d]^{\pi_K\circ a^Q_R} \\
G_0/G_1 \ar[rr]_{\overline{Q\circ P}} & & K_0/K_1 \\
}$$

To show that $\overline{Q\circ P}=\bar{Q}\circ\bar{P}$ it is enough to show that for any $(p,q)\in P\times_{H_0}Q$, we have $$\bar{Q}\circ\bar{P}(\pi_G\circ a_L^P(p))=\pi_K\circ a_R^Q(q).$$  But this reduces to showing that $$\pi_H\circ a_R^P(p)=\pi_H\circ a_L^Q(q),$$  and this is automatic by definition of $P\times_{H_0}Q$.  We have shown that $F$ is a functor.

Now, let $G_1\toto G_0$ and $H_1\toto H_0$ be effective proper \'etale Lie groupoids.  Then $X_G:=G_0/G_1$ and $X_H:=H_0/H_1$ are naturally equipped with orbifold atlases.  Assume that $(X_G,\CIN(X_G))$ and $(X_H,\CIN(X_H))$ are diffeomorphic as differential spaces.  Without loss of generality, we may identify the underlying sets via this diffeomorphism.  By the \hyperref[t:main]{Main Theorem}, the orbifold atlases for $X_G$ and $X_H$ can be reconstructed from $\CIN(X_G)$ and $\CIN(X_H)$, and these orbifold atlases are equivalent since they are constructed out of the same invariants of differential spaces.  We conclude that $G_1\toto G_0$ and $H_1\toto H_0$ are Morita equivalent; that is, isomorphic in $\orb$.
\end{proof}

\begin{remark}\labell{r:arrows}
\noindent
\begin{enumerate}
\item Note that while composition of bibundles is only weakly associative, by \hypref{l:arrows}{Lemma} $F$ carries this weak associativity to true associativity.
\item $F$ is neither full nor faithful.  See \hypref{x:not full nor faithful}{Example}.
\item In the proof above, we did not use the fact that the Lie groupoids are effective proper \'etale to show that $F$ is a functor; this was only used to show that $F$ is essentially injective.  Indeed, $F$ is a restriction of a functor from the weak 2-category of Lie groupoids to differential spaces.
\end{enumerate}
\end{remark}

\section{Proof of Theorem B}\labell{s:diffeology}

This section is designed for readers with some familiarity with the category $\diffeol$ of diffeological spaces.  The main resource on diffeology is the book by Iglesias-Zemmour \cite{iglesias}, although for purposes of this section regarding diffeological orbifolds, the required details appear in \cite{IZKZ}.  The purpose of this section is as follows.  Iglesias-Zemmour, Karshon, and Zadka in \cite{IZKZ} prove that there is a one-to-one correspondence between orbifolds in the classical sense and diffeological orbifolds.  We tailor this result into a functor $G:\orb\to\diffeol$ that is essentially injective on objects, which is \hyperref[t:IZKZ]{Theorem B}.  We give two proofs that this functor is essentially injective.  The first comes directly from the result of Iglesias-Zemmour, Karshon, and Zadka.  For the second, we introduce a functor $\mathbf{\Phi}:\diffeol\to\diffsp$ that sends a diffeological space to its underlying set equipped with the ring of diffeologically smooth real-valued functions.  We show that $F=\mathbf{\Phi}\circ G$.  By the \hyperref[t:main]{Main Theorem} $F$ is essentially injective, and so it follows that $G$ is as well.  The functor $\mathbf{\Phi}$ is studied in \cite[Chapter 2]{watts}, as well as \cite{BIZKW}, and more details about it can be found there.

\begin{definition}[Diffeological Orbifold]\labell{d:diffeol orb}
A \emph{diffeological orbifold} is a diffeological space that is locally diffeologically diffeomorphic to quotient diffeological spaces of the form $\RR^n/\Gamma$, where $\Gamma\subset\GL(\RR^n)$ is a finite subgroup.  (see \cite[Definition 6]{IZKZ}.)
\end{definition}

\begin{proof}[Proof of Theorem B]
Similar to the functor $F:\orb\to\diffsp$ defined in Theorem A, $G$ is the restriction of a functor from the weak 2-category of Lie groupoids with bibundles as arrows and isomorphisms of bibundles as 2-arrows to diffeological spaces.  See \cite[Section 4]{watts-gpds} for details on this functor between Lie groupoids and diffeological spaces.

The fact that $G$ is essentially injective follows from the result of Iglesias-Zemmour - Karshon - Zadka (see Proposition 38, Theorem 39 and Theorem 46 of \cite{IZKZ}) and from \hypref{r:groupoid}{Remark}.
\end{proof}

The functor $G$ is neither faithful nor full, as the following example illustrates.

\begin{example}\labell{x:not full nor faithful}
These examples are due to Iglesias-Zemmour, Karshon, and Zadka and appear as Examples 24 and 25 of \cite{IZKZ}.  Let $\rho_n:\RR\to[0,1]$ be a smooth function with non-empty support inside $[\frac{1}{n+1},\frac{1}{n}]$.  Let $\sigma=(\sigma_1,\sigma_2,...)\in\{-1,1\}^\NN$, and define $f_\sigma:\RR\to\RR$ to be the smooth function $$f_\sigma(x):=\begin{cases}
\sigma_n e^{-1/x}\rho_n(x) & \text{ if $\frac{1}{n+1}<x\leq\frac{1}{n}$ with }n\in\NN,\\
0 & \text{ if $x>1$ or $x\leq 0$}.\\
\end{cases}$$
For any $\sigma$, the function $f_\sigma$ descends to the same diffeologically smooth function $f:\RR\to\RR/\ZZ_2$ (where $\ZZ_2$ acts by reflection).  Thus, the functor $G$ is not faithful.

Next, set $r=\sqrt{x^2+y^2}$ for $(x,y)\in\RR^2$.  Define $g:\RR^2\to\RR^2$ to be the smooth function $$g(x,y):=\begin{cases}
e^{-r}\rho_n(r)(r,0) & \text{ if $\frac{1}{n+1}<r\leq\frac{1}{n}$ and $n$ is even},\\
e^{-r}\rho_n(r)(x,y) & \text{ if $\frac{1}{n+1}<r\leq\frac{1}{n}$ and $n$ is odd},\\
0 & \text{ if $r>1$ or $r=0$ }.\\
\end{cases}
$$  Then, for any integer $k\geq 2$, the function $g$ descends to a diffeologically smooth function $\bar{g}:\RR^2/\ZZ_k\to\RR^2/\ZZ_k$ (where $\ZZ_k$ acts by rotation).  While $\bar{g}$ has a smooth lift $\RR^2\to\RR^2$, this lift is $h_n$-equivariant when restricted to the annulus $\frac{1}{n+1}<r\leq \frac{1}{n}$, where $h_n:\ZZ_k\to\ZZ_k$ is a group homomorphism.  In particular, if $n$ is even, then $h_n$ is the trivial homomorphism; whereas if $n$ is odd, then $h_n$ must be the identity.  Thus, there certainly is no functor, nor even a bibundle, between the groupoid $\ZZ_k\times\RR^2\toto\RR^2$ and itself that corresponds to $f$.  Thus $G$ is not full.
\end{example}

\begin{definition}[The Functor $\mathbf{\Phi}$]\labell{d:phi}
Let $(X,\mathcal{D})$ be a diffeological space.  Define $\Phi\mathcal{D}$ by $$\Phi\mathcal{D}:=\{f:X\to\RR\mid f\circ p\text{ is smooth }\forall p\in\mathcal{D}\}.$$  Then, $\Phi\mathcal{D}$ is a differential structure on $X$ (see \cite[Lemma 2.42]{watts}).  This extends to a functor $\mathbf{\Phi}:\diffeol\to\diffsp$ which sends diffeologically smooth maps to themselves (see the proof of \cite[Theorem 2.48]{watts}).  Note that $\Phi\mathcal{D}$ is just the ring of diffeologically smooth functions of $(X,\mathcal{D})$.
\end{definition}

\begin{proposition}[$F=\mathbf{\Phi}\circ G$]\labell{p:factor through diffeol}
The functor $F:\orb\to\diffsp$ is equal to the composition $\mathbf{\Phi}\circ G$.
\end{proposition}

\begin{proof}
We need only show that given a diffeological orbifold $(X,\mathcal{D})$, that $\Phi\mathcal{D}$ is equal to the orbifold differential structure on $X$.  First, we note that the topology induced by the diffeology on $X$ is equal to the standard orbifold topology (see \cite[Article 2.12]{iglesias}), which in turn is equal to the functional topology induced by $\CIN(X)$ (see \hypref{i:topology}{Item} of \hypref{p:diff orb}{Proposition}).

Now, from Proposition 38, Theorem 39, and Theorem 46 of \cite{IZKZ} we have that the local diffeomorphisms defining the diffeological orbifold structure are exactly the charts of the corresponding orbifold in the sense of \hypref{d:orbifold}{Definition}.  Let $f\in\Phi\mathcal{D}$.  Then, locally where $(X,\mathcal{D})$ is diffeologically diffeomorphic to $\RR^n/\Gamma$, it follows from the definition of a quotient diffeology that $f$ will restrict and lift to a $\Gamma$-invariant function on $\RR^n$.   But this is exactly the pullback of $f$ via an orbifold chart.  Thus, $f\in\CIN(X)$.  In the reverse direction, if $f\in\CIN(X)$, then locally at a chart of the form $(\RR^n,\Gamma,\phi)$, which exists at every point by \hypref{r:isotropy group}{Remark}, we have $f$ restricts and lifts to a $\Gamma$-invariant function on $\RR^n$.  Hence, since the quotient map is a plot of $\mathcal{D}$, it descends to a (local) diffeologically smooth function; that is, a function in $\Phi\mathcal{D}$.  Since smoothness is a local property, the result follows.
\end{proof}

\begin{corollary}\labell{c:factor through diffeol}
$G$ is an essentially injective functor.
\end{corollary}

\begin{proof}
This is immediate from \hypref{p:factor through diffeol}{Proposition} and the fact that $F$ is an essentially injective functor (due to the \hyperref[t:main]{Main Theorem}).
\end{proof}

\begin{remark}\labell{r:factor through diffeol}
\noindent
\begin{enumerate}
\item  In general, the functor $\mathbf{\Phi}:\diffeol\to\diffsp$ is not injective on objects, as the example below illustrates.  Also, while it is faithful, it is not full (see, for example, the end of Example 2.67 in \cite{watts}).  It remains an open question whether or not $\mathbf{\Phi}$ restricted to diffeological orbifolds is full.
\item Since $G$ is neither faithful nor full (see \hypref{x:not full nor faithful}{Example}), it follows from \hypref{p:factor through diffeol}{Proposition} that $F$ is neither faithful nor full.
\end{enumerate}
\end{remark}

\begin{example}[Rotations of $\RR^n$]\labell{x:rotations}
Let $O(n)$ act on $\RR^n$ by rotations about the origin.  Then the quotient diffeology $\mathcal{D}_n$ on $\RR^n/O(n)$ depends on $n$ (see Exercise 50 of \cite{iglesias} with solution at the back of the book).  The corresponding quotient differential structure which is equal to $\Phi\mathcal{D}_n$, however, is equal to $\CIN([0,\infty))$, the subspace differential structure of $[0,\infty)\subset\RR$.
\end{example}


\end{document}